\documentclass[reqno]{amsart}
\usepackage{hyperref}

\begin{document}
\title[A criterion of solvability of the elliptic Cauchy problem]
{A criterion of solvability of the elliptic Cauchy problem in a
multi-dimensional cylindrical domain }

\author[T.\,Sh.~Kalmenov, M.\,A.~Sadybekov and
B.\,T.~Torebek\hfil \hfilneg] {Tynysbek~Sh.~Kalmenov,
Makhmud~A.~Sadybekov and
Berikbol~T.~Torebek}  % in alphabetical order

\address{Tynysbek~Sh.~Kalmenov \newline
Department of Differential Equations, \newline Institute of
Mathematics and Mathematical Modeling.\newline 125 Pushkin str.,
050010 Almaty, Kazakhistan} \email{kalmenov.t@mail.ru}

\address{Makhmud~A.~Sadybekov \newline
Department of Differential Equations, \newline Institute of
Mathematics and Mathematical Modeling.\newline 125 Pushkin str.,
050010 Almaty, Kazakhistan} \email{sadybekov@math.kz}

\address{Berikbol T. Torebek \newline
Department of Differential Equations, \newline Institute of
Mathematics and Mathematical Modeling.\newline 125 Pushkin str.,
050010 Almaty, Kazakhistan} \email{torebek@math.kz}

\subjclass[2000]{31A30, 31B30, 35J40} \keywords{elliptic Cauchy
problem, self-adjoint operator, elliptic operator with deviating
argument}

\begin{abstract}
In this paper we consider the Cauchy problem for multidimensional
elliptic equations in a cylindrical domain. The method of spectral
expansion in eigenfunctions of the Cauchy problem for equations
with deviating argument establishes a criterion of the strong
solvability of the considered elliptic Cauchy problem. It is shown
that the ill-posedness of the elliptic Cauchy problem is
equivalent to the existence of an isolated point of the continuous
spectrum for a self-adjoint operator with deviating
argument.\end{abstract}

\maketitle \numberwithin{equation}{section}
\newtheorem{theorem}{Theorem}
\newtheorem{corollary}{Corollary}
\newtheorem{lemma}{Lemma}
\newtheorem{remark}{Remark}
\newtheorem{problem}{Problem}
\newtheorem{example}{Example}
\newtheorem{definition}{Definition}
\allowdisplaybreaks

\section{Introduction}\label{sec1}
As it is known, the solution of the Cauchy problem for the Laplace
equation is unique, but unstable. First of all it should be noted
that the existence and uniqueness of its solution is essentially
guaranteed by the universal Cauchy - Kovalevskaja theorem, which
holds for Elliptic Problems. However, the existence of the
solution is guaranteed only in a small. Traditionally the
ill-posedness of the elliptic Cauchy problem is determined in
relation to its equivalence to Fredholm integral equations of the
first kind. The problem of solving the operator equation of the
first kind can not be correct, since the operator which is inverse
to completely continuous operator is not continuous. The Cauchy
problem for the Laplace equation is one of the main examples of
ill-posed problems. One can pick up the harmonic functions with
arbitrarily small Cauchy data on a piece of the domain boundary,
which will be arbitrarily large in the domain (the famous example
of Hadamard) \cite{hadamard23}. For the formulation of the problem to be
correct, it is necessary to narrow down the class of solutions.
The stability of a plane problem in the class of bounded solutions
firstly was proved by Carleman \cite{carleman26}. From Carleman's results
immediately follow estimations characterizing this stability. In
the mentioned work Carleman obtained a formula for determinating
an analytic function of a complex variable by its values on some
piece of the arc. However, this formula is unstable and therefore
can not be directly used as an efficient method. The first results
related to the construction of an efficient algorithm for solving
the problem, best of our knowledge, are published simultaneously
in works Carlo Pucci \cite{pucci55} and M.M. Lavrentev \cite{lavrentev55}. Estimates
characterizing the stability of a spatial problem in the class of
bounded solutions, were first obtained by M.M. Lavrentev \cite{lavrentev55} for
harmonic functions, given in a straight cylinder and vanishing on
the generators. The Cauchy data were given on the base of the
cylinder. A little later, similar estimates were obtained by S.N.
Mergelyan \cite{mergelyan56} for the functions within a sphere and by M.M.
Lavrentev \cite{lavrentev57} for an arbitrary spatial domain with sufficiently
smooth boundary. Around the same time, E.M. Landis \cite{landis63} obtained
estimates characterizing the stability of spatial problem for an
arbitrary elliptic equation. The above results laid the foundation
for the theory of ill-posed Cauchy problems for elliptic
equations. By now this theory has deep development both for the
plane, and for the spatial cases, and also for general elliptic
equations of high order, etc. Methods of regularization and
solutions of ill-posed problems have been proposed in \cite{tikhonovarsenin79}-\cite{klibanov15}. In
these works the concept of conditional correctness of such
problems is introduced and algorithms for constructing their
solutions are proposed. In contrast to the presented results, in
this paper a new criterion of well-posedness (ill-posedness)
initial boundary value problem for a general an unbounded equation
is proved. The principal difference of our work from the work of
other authors is the application of spectral problems for
equations with deviating argument in the study of ill posed
boundary value problems. The present work is an extension of
results \cite{kalmenoviskakova07}-\cite{kalmenoviskakova09} on the case of more general unbounded operators in
a multidimensional cylindrical domain.

\section{Formulation of the problem and main results}
Let $ \Omega  \subset R^n $ be a bounded domain with smooth
boundary $\partial \Omega $ and $ D = \Omega  \times \left( {0,1}
\right) $ is a cylinder. Let $ \mathcal{L}_x $ be a self-adjoint unbounded
operator in $ L{}_2\left( \Omega  \right) $ with compact resolvent
satisfying the Friedrichs type inequality $ \left( {\mathcal{L}_x u,u}
\right) \ge \left\| u \right\|^2 .$ By $ \mu _k $ we denote all
eigenvalues (numbered in decreasing order) and by $ u_k \left( x
\right),k \in N $ we denote a complete system of all orthonormal
eigenfunctions of the operator $ \mathcal{L}_x $ in $ L{}_2\left( \Omega
\right) .$ In $ D $ we consider a mixed Cauchy problem for
elliptic equations \begin{equation}\label{1}\mathbb{L}u \equiv u_{tt}
\left( {t,x} \right) - \mathcal{L}_x u\left( {t,x} \right) = f\left( {t,x}
\right),\left( {t,x} \right) \in D,\end{equation} with initial
conditions \begin{equation}\label{2}u\left( {0,x} \right) = u_t
\left( {0,x} \right) = 0,x \in \Omega.\end{equation} and for every
$ t \in \left( {0,1} \right) $ satisfying a condition $ u \in
\mathcal{D}\left( {\mathcal{L}_x } \right).$

\begin{definition} The function $ u \in L_2 (D) $
we will call a strong solution of the mixed Cauchy problem
\eqref{1}, \eqref{2}, if there exists a sequence of functions $
u_n  \in C^2 \left( {\bar D} \right) $ satisfying conditions
\eqref{2} and (for every $ t \in \left( {0,1} \right) $ )
belonging to $ \mathcal{D}\left( {\mathcal{L}_x } \right) ,$ such that $ u_n $ and $
Lu_n $ converge in the norm  $ L_2 (D) $ respectively to $ u $ and
$ f .$\end{definition}

In the future, the following eigenvalue problem for an elliptic
equation with deviating argument will play an important role.

Find numerical values of $ \lambda $ (eigenvalues), under which
the problem for a differential equation with deviating argument
\begin{equation}\label{3} \mathbb{L}u \equiv u_{tt} \left( {x,t} \right) -
\mathcal{L}_x u\left( {x,t} \right) = \lambda u\left( {x,1 - t}
\right),\left( {x,t} \right) \in D,\end{equation} has nonzero
solutions (eigenfunctions) satisfying conditions \eqref{2}.
Obviously, the equivalent record of equation \eqref{3} has the form
$$\mathbb{L}Pu \equiv P\left( {u_{tt} \left( {t,x} \right) - \mathcal{L}_x u\left(
{t,x} \right)} \right) = \lambda u\left( {t,x} \right),\left(
{t,x} \right) \in D, $$ where $ Pu(t,x) = u(1 - t,x) $ is a
unitary operator.
\begin{theorem}\label{th1.1} The spectral Cauchy problem \eqref{3}, \eqref{2} has a
complete orthonormal system of eigenfunctions
\begin{equation}\label{4}u_{km} (x,t) = u_k \left( x \right) \cdot
v_{km} \left( t \right),\end{equation} where $ k,m \in
\mathbb{N},$ $v_{km} (t) $ are non-zero solutions of the problem
\begin{equation}\label{5} v''_{km} \left( t \right) - \mu _k
v_{km} \left( t \right) = \lambda _{km} v_{km} \left( {1 - t}
\right),{\rm{ }}0 < t < 1,\end{equation} \begin{equation}\label{6}
v_{km} \left( 0 \right) = v'_{km} \left( 0 \right) =
0,\end{equation} and $ \lambda _{km}$ are eigenvalues of problem
\eqref{3}, \eqref{2}. In addition for the large $ k $ the smallest
eigenvalue $ \lambda _{k1} $ has the asymptotic behavior
\begin{equation}\label{7}\lambda _{k1}  = 4\mu_k e^{ - \sqrt{\mu_k}} \left(
{1 + o(1)} \right)\,.\end{equation}\end{theorem}

\begin{theorem}\label{th1.2} A strong solution of the mixed Cauchy problem \eqref{1} -
\eqref{2} exists if and only if $ f\left( {x,\,t} \right) $
satisfies the inequality \begin{equation}\label{8} \sum\limits_{k
= 1}^\infty  {\left| {\frac{{\tilde f_{k1} }}{{\lambda _{k1} }}}
\right|} ^2  < \infty ,\end{equation} where $ \tilde f_{km}  =
\left( {f(x,\,1 - t),\,u_{km} (x,\,t)} \right) .$

If condition \eqref{8} holds, then a solution of problem \eqref{1}
- \eqref{2} can be written as \begin{equation}\label{9} u\left(
{x,t} \right) = \sum\limits_{k = 1}^\infty {\frac{{\tilde f_{k1}
}}{{\lambda _{k1} }}u_{k1} \left( {x,\,t} \right)}  +
\sum\limits_{k = 1}^\infty {\sum\limits_{m = 2}^\infty
{\frac{{\tilde f_{km} }}{{\lambda _{km} }}u_{km} } } \left( {x,t}
\right).
\end{equation}\end{theorem} By $ \tilde L_2 \left( D \right) $ we
denote a subspace of $ L_2 (D),$ spanned by the eigenvectors $$
\left\{ {u_{k1} \left( {x,t} \right)} \right\}_{k = p  +
1}^\infty, p \in \mathbb{N} $$ and by $ \hat L_2 \left( D \right)
$ we denote its orthogonal complement $$ L_2 \left( D \right) =
\tilde L_2 \left( D \right) \oplus \hat L_2 \left( D \right) .$$
\begin{theorem}\label{th1.3} For any $ f \in \hat L_2 \left( D
\right) $ a solution of the problem \eqref{1} - \eqref{2} exists,
is unique and belongs to $ \hat L_2 \left( D \right).$ This
solution is stable and has the form
\begin{equation}\label{10} u\left( {x,t} \right) = \sum\limits_{k =
1}^p {\frac{{\tilde f_{k1} }}{{\lambda _{k1} }}u_{k1} \left(
{x,\,t} \right)}  + \sum\limits_{k = 1}^\infty {\sum\limits_{m =
2}^\infty {\frac{{\tilde f_{km} }}{{\lambda _{km} }}u_{km} } }
\left( {x,t} \right).\end{equation}\end{theorem}

\section{Some auxiliary statements}
In this section we present some auxiliary results to prove the
main results. \begin{lemma}\label{lem1.1} For each fixed value of the index $ k $ the spectral problem \eqref{5} - \eqref{6} has a complete orthonormal in $L_2 \left( {0,1} \right)$ system of eigenfunctions $v_{km} \left( t \right),{\rm{ }}m \in \mathbb{N}, $ corresponding to the eigenvalues $\lambda _{km}.$

These eigenvalues $ \lambda _{km} $ are roots of the equation \begin{equation}\label{11}\sqrt {\mu _k  - \lambda } ch\frac{{\sqrt {\mu _k  + \lambda }}}{2}ch\frac{{\sqrt {\mu _k  - \lambda } }}{2} - \sqrt {\mu _k  +\lambda } sh\frac{{\sqrt {\mu _k  + \lambda } }}{2}sh\frac{{\sqrt{\mu _k  - \lambda } }}{2} = 0.\end{equation}\end{lemma}
\begin{proof} Indeed, applying an inverse operator $ L_C^{ - 1} $ to the Cauchy problem \eqref{5} - \eqref{6} we arrive at the operator equation $$ v_{km} \left( t \right) = \lambda L_C^{ - 1} Pv_{km} \left( t\right), $$ where $ Pf\left( t \right) = f\left( {1 - t} \right), $ and a function  $\phi \left( t \right) = L_C^{ - 1} f\left( t \right) $ is the solution of the Cauchy problem $$ \phi ''\left( t \right) - \mu _k \phi \left( t \right) = f\left(t \right),\phi \left( 0 \right) = \phi '\left( 0 \right) =0,\forall f\left( t \right) \in L_2 \left( {0,1} \right). $$
Then for the operator $ L_C^{ - 1} $ we have the representation \begin{equation}\label{12}L_C^{ - 1} f\left( t \right) = \frac{1}{{\sqrt {\mu _k }}}\int\limits_0^t {f\left( \xi  \right)sh\sqrt {\mu _k } \left( {t- \xi } \right)d\xi } ,\forall {\rm{ }}f\left( t \right) \in L_2\left( {0,1} \right).\end{equation}
Therefore, the adjoint to $ L_C^{ - 1} $ operator has the form \begin{equation}\label{13}\left( {L_C^{ - 1} } \right)^ *  f\left( t \right) =\frac{1}{{\sqrt {\mu _k } }}\int\limits_t^1 {f\left( \xi\right)sh\sqrt {\mu _k } \left( {\xi  - t} \right)d\xi } ,\forall {\rm{ }}f\left( t \right) \in L_2 \left( {0,1} \right).\end{equation}
Taking into account representation \eqref{12} and \eqref{13}, it is easy to make sure that $$ L_C^{ - 1} Pf = P\left( {L_C^{ - 1} } \right)^ *f. $$

Then the chain of equalities $$ L_C^{ - 1} Pf = P\left( {L_C^{ - 1}} \right)^* f = P^*  \left( {L_C^{ - 1} } \right)^*f =\left( {L_C^{ - 1} P} \right)^* f,\forall {\rm{ }}f\left( t\right) \in L_2 \left( {0,1} \right), $$ allows us to conclude that the operator $ L_C^{ - 1} P $ is completely continuous self-adjoint Hilbert-Schmidt operator \cite{gohberg03}. Therefore for each $k\in \mathbb{N}, $ the spectral problem \eqref{5} - \eqref{6} has a complete orthonormal system of functions $v_{km} \left( t \right),\,m\in \mathbb{N}\,$ in $L_2 \left( {0,1} \right). $

Since $ \mathcal{L}_x $ is a self-adjoint positive definite unbounded operator, then all its eigenvalues are real and positive, and system of eigenfunctions forms a complete orthonormal system in $L{}_2\left( \Omega  \right) $ \cite{naimark68}.

We are looking for eigenfunctions of problem \eqref{3}, \eqref{2} by means of the Fourier method of separation of variables in the form $$ u_k \left( {x,t} \right) = u_k \left( x \right)v\left( t\right), $$ where $ k \in \mathbb{N}.$ Therefore, for determination of unknown function $ v\left( t \right) $ we get the spectral problem \eqref{5}, \eqref{6} for an equation with deviating argument.\\
It is easy to show that the general solution of equation \eqref{5} has the form $$ v\left( t \right) = c_1 {\rm{ }}ch\sqrt {\mu _k  +\lambda } \left( {t - \frac{1}{2}} \right) + c_2 {\rm{ }}sh\sqrt
{\mu _k  - \lambda } \left( {t - \frac{1}{2}} \right) ,$$ where  $c_1 $ and $c_2 $ are some constants.\\
Using the initial conditions \eqref{6}, we arrive at the system of linear homogeneous equations concerning these constants.  As we know, this system has a nontrivial solution if a determinant of the system $$ \Delta (\lambda ) = \left| {\begin{array}{*{20}c}
   {ch\frac{{\sqrt {\mu _k  + \lambda } }}{2}} & {sh\frac{{\sqrt {\mu _k  - \lambda } }}{2}}  \\
   {\sqrt {\mu _k  + \lambda } sh\frac{{\sqrt {\mu _k  + \lambda } }}{2}} & {\sqrt {\mu _k  - \lambda } ch\frac{{\sqrt {\mu _k  - \lambda } }}{2}}  \\
\end{array}} \right|$$ is zero. Thus, for determining the parameter $ \lambda $ we get \eqref{11}.
The proof of Lemma \ref{lem1.1} is complete. \end{proof}

Let \begin{equation}\label{14} \varpi _k \left( \lambda
\right) = \ln \left( {cth\frac{{\sqrt {\mu _k  + \lambda } }}{2}}
\right) + \ln \left( {cth\frac{{\sqrt {\mu _k  - \lambda } }}{2}}
\right) - \frac{1}{2}\ln \left( {\frac{{\mu _k  + \lambda }}{{\mu
_k  - \lambda }}} \right) = 0.\end{equation}

\begin{lemma}\label{lem1.2} There exists a number $ \lambda _0 $ such that for all $$ 0 < \lambda  < \lambda _0  < \frac{{\mu _k }}{{4\mu _k  + \theta }},\,k \ge 1,\,\theta  \in \left( {0,\,1} \right)\,, $$ the following statements are true:\\
1) the function $ \varpi '_k\left( \lambda  \right) $ is a constant sign;\\
2) for the function $ \varpi ''_k \left( \lambda  \right) $ the following inequality holds $\left| {\lambda \mu _k \varpi ''_k \left( \lambda  \right)}\right| < 1,\,k > 1\,. $\end{lemma}

\begin{proof} By virtue of Lemma \ref{lem1.1} we have the real eigenvalues of problem \eqref{5}-\eqref{6}, that is, real roots $ \lambda _{km} $ of equation \eqref{11}. It is easy to verify that $ \lambda _{km}  > 0\,.$ Indeed, let us write the asymptotic behavior of the smallest eigenvalues $\lambda _{km} $ at $k \to \infty.$

After a nontrivial transformation of equation \eqref{11}, we have \begin{equation}\label{15}\frac{{\sqrt {\mu _k  + \lambda } }}{{\sqrt {\mu _k  - \lambda }}} = cth\frac{{\sqrt {\mu _k  + \lambda } }}{2}cth\frac{{\sqrt{\mu _k  - \lambda } }}{2}.\end{equation}
Assuming  $ \left| \lambda  \right| < 1 $ and logarithming both sides of equation \eqref{15}, we obtain \eqref{14}. By calculating the derivative $ \varpi _k \left( \lambda  \right),$ we get $ \varpi '_k \left( 0 \right) =  - \frac{1}{{\mu _k }}.$

Then the required boundary of monotonicity of $\varpi _k (\lambda ) $ can be determined from the relation $$ \varpi '_k (\lambda _0 ) = \varpi '_k (0) + \varpi ''_k \left(
{\theta \lambda _0 } \right)\lambda _0  < 0\,,$$ where $\lambda _0 \,:\,\,0 < \lambda _0  < 1\,,\, $ and $ \theta\in {\rm{(0}}{\rm{, 1)}} $ is an arbitrary number. Thus, for determining $\lambda _0 $ we have the condition \begin{equation}\label{16}\lambda _0 \mu _k \varpi ''_k \left( {\theta \lambda _0 }\right) < 1.\end{equation}
We write explicitly the second derivative of functions  $ \varpi_k (\lambda ):$
\begin{align*}\varpi ''_k (\lambda ) \\&= \frac{{ch\sqrt {\mu _k  + \lambda }
}}{{4\left( {\mu _k  + \lambda } \right)sh^2 \sqrt {\mu _k  +
\lambda } }} + \frac{{ch\sqrt {\mu _k  - \lambda } }}{{4\left(
{\mu _k  - \lambda } \right)sh^2 \sqrt {\mu _k  - \lambda } }}\\&
+ \frac{1}{{4\sqrt {\left( {\mu _k  + \lambda } \right)^3 }
sh\sqrt {\mu _k  + \lambda } }} + \frac{1}{{4\sqrt {\left( {\mu _k
- \lambda } \right)^3 } sh\sqrt {\mu _k  - \lambda } }} -
\frac{{2\lambda \mu _k }}{{\left( {\mu _k^2  - \lambda ^2 }
\right)^2 }}\\&
= \frac{{e^{ - 3\sqrt {\mu _k  + \lambda } } \left( {1 + e^{ - 2\sqrt {\mu _k  + \lambda } } } \right)}}{{2\left( {\mu _k  + \lambda } \right)\left( {1 - e^{ - 2\sqrt {\mu _k  + \lambda } } } \right)^2 }} + \frac{{e^{ - 3\sqrt {\mu _k  - \lambda } } \left( {1 + e^{ - 2\sqrt {\mu _k  - \lambda } } } \right)}}{{2\left( {\mu _k  - \lambda } \right)\left( {1 - e^{ - 2\sqrt {\mu _k  - \lambda } } } \right)^2 }}\\&
+ \frac{{e^{ - 2\sqrt {\mu _k  + \lambda } } }}{{4\sqrt
{\left( {\mu _k  + \lambda } \right)^3 } \left( {1 - e^{ - 2\sqrt
{\mu _k  + \lambda } } } \right)}} + \frac{{e^{ - 2\sqrt {\mu _k
- \lambda } } }}{{4\sqrt {\left( {\mu _k  - \lambda } \right)^3 }
\left( {1 - e^{ - 2\sqrt {\mu _k  - \lambda } } } \right)}}\\& -
\frac{{2\lambda \mu _k }}{{\left( {\mu _k^2  - \lambda ^2 }
\right)^2 }} .\end{align*} Further \begin{align*}\varpi ''_k \left( {\lambda _0 \theta }
\right)\\& =\frac{{ch\sqrt {\mu _k  + \lambda _0 \theta }
}}{{4\left( {\mu _k  + \lambda _0 \theta } \right)sh^2 \sqrt {\mu
_k  + \lambda _0 \theta } }} + \frac{{ch\sqrt {\mu _k  - \lambda
_0 \theta } }}{{4\left( {\mu _k  - \lambda _0 \theta } \right)sh^2
\sqrt {\mu _k  - \lambda _0 \theta } }}\\&
+ \frac{1}{{4\sqrt {\left( {\mu _k  + \lambda _0 \theta }
\right)^3 } sh\sqrt {\mu _k  + \lambda _0 \theta } }} +
\frac{1}{{4\sqrt {\left( {\mu _k  - \lambda _0 \theta } \right)^3
} sh\sqrt {\mu _k  - \lambda _0 \theta } }}\\& - \frac{{2\lambda _0
\theta \mu _k }}{{\left( {\mu _k^2  - \left( {\lambda _0 \theta }
\right)^2 } \right)^2 }}.\end{align*} As $$ \frac{{2\lambda _0 \theta \mu _k
}}{{\left( {\mu _k^2  - \left( {\lambda _0 \theta } \right)^2 }
\right)^2 }} =  \frac{1}{2}\left(
{\frac{1}{{\left( {\mu _k^{}  - \lambda _0 \theta } \right)^2 }} -
\frac{1}{{\left( {\mu _k^{}  + \lambda _0 \theta } \right)^2 }}}
\right) \ge  - \frac{1}{{\left( {\mu _k^{}  + \lambda _0 \theta }
\right)^2 }}$$
and \begin{align*}\frac{{ch\sqrt {\mu _k  + \lambda _0 \theta } }}{{sh^2 \sqrt
{\mu _k  + \lambda _0 \theta } }}\\& = \frac{1}{2}\left( {\frac{1}{{ch\sqrt {\mu _k  + \lambda _0
\theta }  - 1}} + \frac{1}{{ch\sqrt {\mu _k  + \lambda _0 \theta }
+ 1}}} \right)\\& \le \frac{1}{{ch\sqrt {\mu _k  + \lambda _0 \theta
}  - 1}},\end{align*} \begin{align*}\frac{{ch\sqrt {\mu _k  - \lambda _0 \theta }
}}{{sh^2 \sqrt {\mu _k  - \lambda _0 \theta } }}\\&=\frac{1}{2}\left( {\frac{1}{{ch\sqrt {\mu _k
- \lambda _0 \theta }  - 1}} + \frac{1}{{ch\sqrt {\mu _k  -
\lambda _0 \theta }  + 1}}} \right)\\& \le \frac{1}{{ch\sqrt {\mu _k
- \lambda _0 \theta }  - 1}}.\end{align*} Then the following inequality is true: \begin{align*}\varpi ''_k
(\lambda _0 \theta )\\& \le \frac{1}{{4\left( {\mu _k  + \lambda _0
\theta } \right)\left( {ch\sqrt {\mu _k  + \lambda _0 \theta }  -
1} \right)}} + \frac{1}{{4\left( {\mu _k  - \lambda _0 \theta }
\right)\left( {ch\sqrt {\mu _k  - \lambda _0 \theta }  - 1}
\right)}}\\&
+ \frac{1}{{4\sqrt {\left( {\mu _k  + \lambda _0 \theta }
\right)^3 } sh\sqrt {\mu _k  + \lambda _0 \theta } }} \\&+
\frac{1}{{4\sqrt {\left( {\mu _k  - \lambda _0 \theta } \right)^3
} sh\sqrt {\mu _k  - \lambda _0 \theta } }} + \frac{1}{{\left(
{\mu _k^{}  + \lambda _0 \theta } \right)^2 }} \\&
\le \frac{1}{{2\left( {\mu _k  - \lambda _0 \theta } \right)\left( {ch\sqrt {\mu _k  - \lambda _0 \theta }  - 1} \right)}}\\& + \frac{1}{{2\sqrt {\left( {\mu _k  - \lambda _0 \theta } \right)^3 } sh\sqrt {\mu _k  - \lambda _0 \theta } }} + \frac{1}{{\left( {\mu _k^{}  + \lambda _0 \theta } \right)^2 }}\\& = \frac{{e^{ - \sqrt {\mu _k  - \lambda _0 \theta } } }}{{\left( {\mu _k  - \lambda _0 \theta } \right)\left( {e^{ - 2\sqrt {\mu _k  - \lambda _0 \theta } }  - 2e^{ - \sqrt {\mu _k  - \lambda _0 \theta } }  + 1} \right)}}\\&
+ \frac{{e^{ - \sqrt {\mu _k  - \lambda _0 \theta } } }}{{\sqrt {\left( {\mu _k  - \lambda _0 \theta } \right)^3 } \left( {1 - e^{ - 2\sqrt {\mu _k  - \lambda _0 \theta } } } \right)}} + \frac{1}{{\left( {\mu _k^{}  + \lambda _0 \theta } \right)^2 }}\\&
\le \frac{1}{{\left( {\mu _k  - \lambda _0 \theta } \right)}}\left( {\frac{{e^{ - \sqrt {\mu _k  - \lambda _0 \theta } } }}{{\left( {1 - e^{ - \sqrt {\mu _k  - \lambda _0 \theta } } } \right)^2 }}\, + \frac{{e^{ - \sqrt {\mu _k  - \lambda _0 \theta } } }}{{\left( {1 - e^{ - 2\sqrt {\mu _k  - \lambda _0 \theta } } } \right)}} + 1} \right)\\&
\le \frac{1}{{\left( {\mu _k  - \lambda _0 \theta } \right)}}\frac{{2 + \left( {1 - e^{ - \sqrt {\mu _k  - \lambda _0 \theta } } } \right)^2 }}{{\left( {1 - e^{ - \sqrt {\mu _k  - \lambda _0 \theta } } } \right)^2 }}.\end{align*}
Hence \begin{equation}\label{17}\varpi ''_k \left( {\lambda _0 \theta } \right) <
\frac{1}{{\left( {\mu _k  - \lambda _0 \theta } \right)}}\frac{{2
+ 1 - 2e^{ - \sqrt {\mu _k  - \lambda _0 \theta } }  + e^{ -
2\sqrt {\mu _k  - \lambda _0 \theta } } }}{{\left( {1 - e^{ -
\sqrt {\mu _k  - \lambda _0 \theta } } } \right)^2 }}\end{equation}
Further, for large values $ k $ from \eqref{17} we obtain the validity of the inequality $$\varpi ''_k (\lambda _0 \theta ) \le \frac{4}{{\mu _k  - \lambda_0 \theta }}.$$
Applying the condition \eqref{16} to the last inequality, we obtain the desired boundary for $ \lambda _0:$ $$\lambda _0  < \frac{{\mu _k }}{{4\mu _k  + \theta }},\,k >1,\,0 < \theta  < 1.$$ Lemma \ref{lem1.2} is proved. \end{proof}

Consider now the question of an asymptotic behavior of the eigenvalues of problem \eqref{5} - \eqref{6} at large $ k.$
\begin{lemma}\label{lem1.3} An asymptotic behavior of eigenvalues of the problem \eqref{5}- \eqref{6}, not exceeding $ \lambda_0,$ for the large values of $ k$ has the form \eqref{7}.\end{lemma}
\begin{proof} According to Lemma \ref{lem1.2} the monotonic function $ \varpi _k\left( \lambda  \right) $ in the interval $\left( {0,\,\lambda _0 } \right) $ can have only one zero. By the Taylor formula we have $$\varpi _k \left( \lambda  \right) = \varpi _k \left( 0 \right) +\frac{{\varpi '_k \left( 0 \right)}}{{1!}}\lambda  + \frac{{\varpi''_k (\theta \lambda )}}{{2!}}\lambda ^2  < 0\,,\,0 < \theta  <1.$$
Substituting the calculated values of the function $ \varpi _k $ and its derivative $\varpi '_k,$ we get $$\varpi _k \left( \lambda  \right) = 2\ln\left( {cth\frac{{\sqrt {\mu _k } }}{2}} \right) - \frac{\lambda}{{\mu _k }} + \varpi ''_k \left( {\theta \lambda }\right)\frac{{\lambda ^2 }}{2}.$$ Then the zero of linear part of the function $$\mu _k \varpi _k \left( \lambda  \right) = 2\mu _k\ln \left( {cth\frac{{\sqrt {\mu _k } }}{2}} \right) - \lambda  +\frac{{\mu _k \lambda ^2 }}{2}\varpi ''_k \left( {\theta \lambda }\right)$$ will be $$\lambda _{k1}  = 2\mu _k \ln \left( {cth\frac{{\sqrt{\mu _k } }}{2}} \right) = 2\mu _k \ln \left( {\frac{{1 + e^{ -\sqrt {\mu _k } } }}{{1 - e^{ - \sqrt {\mu _k } } }}} \right).$$

For sufficiently large values $k \in \mathbb{N},$ considering the asymptotic formulas, $\lambda _{k1} $ can be written as $$\lambda _{k1}  = 4\mu _k e^{ - \sqrt {\mu _k } } \left( {1 +o\left( 1 \right)} \right).$$
Taking into account the result of Lemma \ref{lem1.2} on a circle $$\left|\lambda  \right| = 4\mu _k e^{ - \sqrt {\mu _k } } \left( {1 +\varepsilon } \right),$$ where $ \varepsilon $ is a greatly small positive number, for sufficiently large $k \ge k_0 (\varepsilon ) $ it is easy to check the validity of inequality \begin{align*}\left| {\varpi ''_k \left( {\theta \lambda } \right)\mu _k \lambda ^2 } \right|_{\left| \lambda  \right| = 4\mu _k e^{ - \sqrt {\mu_k } } \left( {1 + \varepsilon } \right)\,} \\& \le C\left| {2\mu_k \ln \left( {\frac{{1 + e^{ - \sqrt {\mu _k } } }}{{1 - e^{ -\sqrt {\mu _k } } }}} \right) - \,\lambda \,}\right|_{\left| \lambda  \right| = 4\mu _k e^{ - \sqrt {\mu _k } }\left( {1 + \varepsilon } \right)}.\end{align*}

Then, by Rouche's theorem \cite{titchmarsh39}, we have that the quantity of zeros
of $\mu _k \varpi _k \left( \lambda  \right) $ and its linear part coincide and are inside the circle $\left| \lambda  \right| = 4\mu _k e^{ - \sqrt {\mu _k } } \left(
{1 + \varepsilon } \right).$ Consequently, the function $ \mu _k\varpi _k \left( \lambda  \right) $ for $0 < \lambda  < \lambda _0 $ has one zero, the asymptotic behavior is given by formula \eqref{7}. Lemma \ref{lem1.3} is proved.\end{proof}

\section{Proof of the main results}

\begin{proof} (Theorem \ref{th1.1}) By $ u_k \left( x \right),k \in \mathbb{N} $ we have denoted a complete system of orthonormal eigenfunctions of the operator $\mathcal{L}_x $ in $L{}_2\left( \Omega  \right).$ By Lemma \ref{lem1.1}, for each fixed value of the $ k $ the spectral problem \eqref{5} - \eqref{6} has complete orthonormal system of eigenfunctions $v_{km} \left( t \right),{\rm{ }}m\in \mathbb{N} $ in $L_2 \left( {0,1} \right).$ Then the system \eqref{4} forms a complete orthogonal system in $L_2 \left( D \right).$ Consequently, problem \eqref{3}, \eqref{2} does not have the other eigenvalues and eigenfunctions. Theorem \ref{th1.1} is proved.\end{proof}
\begin{proof}(Theorem \ref{th1.2}) Let $u(x,t) \in C^2 \left( D \right) $ be a solution of problem \eqref{1} - \eqref{2}. Then, by virtue of the completeness and orthonormality of eigenfunctions $u_{km} (x,t) $ of problem \eqref{3}, \eqref{2}, the function $u(x,t) $ in $L_2 \left( D \right) $ can be expanded in a series \cite{naimark68} \begin{equation}\label{18}u(x,t) = \sum\limits_{k = 1}^\infty  {\sum\limits_{m = 1}^\infty
{a_{km} u_{km} } } (x,t),\end{equation} where $ a_{km} $ are Fourier coefficients of the system $u_{km} (x,t).$ Rewriting equation \eqref{1} in the form \begin{equation}\label{19}LPu =P\left( {u_{tt} (x,t) - \mathcal{L}_x u(x,t)} \right) = Pf(x,t),\end{equation} and substituting the solution of form \eqref{18} in equation \eqref{19} according to the ratio $$P\left( {\frac{{\partial ^2 u_{km} }}{{\partial t^2 }}(x,t) - \mathcal{L}_x u_{km} (x,t)} \right) = \lambda _{km} u_{km} (x,t),$$ we have                                                    $$a_{km}  = \frac{{\tilde f_{km} }}{{\lambda _{km} }},$$ where $\tilde f_{km}  = \left( {f(x,\,1 - t),\,u_{km} (x,\,t)} \right) .$

Thus for solutions $u(x,t) $ we obtain the following explicit representation \begin{equation}\label{20}u(x,t) = \sum\limits_{k = 1}^\infty  {\sum\limits_{m = 1}^\infty
{\frac{{\tilde f_{km} }}{{\lambda _{km} }}u_{km} } } (x,t).\end{equation}

Note that the representation \eqref{20} remains true for any strong solution of problem \eqref{1}-\eqref{2}. We have obtained this representation under the assumption that the solution of the Cauchy problem \eqref{1}-\eqref{2} exists.

The question naturally arises, for what subset of the functions $f \in L_2 \left( D \right) $ there exists a strong solution?

To answer this question, we represent the formula \eqref{20} in the form \eqref{9} from which, by Parseval's equality, it follows \begin{equation}\label{21}\left\|{\rm{u}} \right\|^2  = \sum\limits_{k = 1}^\infty  {\left|{\frac{{\tilde f_{k1} }}{{\lambda _{k1} }}} \right|} ^2  +\sum\limits_{k = 1}^\infty  {\sum\limits_{m = 2}^\infty  {\,\left|{\frac{{\tilde f_{km} }}{{\lambda _{km} }}} \right|^2 } }.\end{equation} By virtue of Lemma \ref{lem1.3} we have $ \lambda _{km}  \ge\frac{1}{4},\,m > 1. $

Therefore, the right-hand side of equality \eqref{21} is limited only for those $f\left( {x,\,t} \right) ,$ for which the weighted norm \eqref{8} is limited. This fact proves Theorem \ref{th1.2}.\end{proof}
\begin{proof} (Theorem \ref{th1.3}) Obviously that the operator $ \mathbb{L} $ is invariant in $\hat L_2 \left( D \right).$ By Theorem \ref{th1.2}, for any $ f \in \hat L_2 \left( D \right) $ there exists a unique solution of problem \eqref{1} - \eqref{2} and it can be represented in the form \eqref{10}.
Therefore, a certain infinite-dimensional space $ \hat L_2 \left(D \right) $ is the space of correctness of the Cauchy problem \eqref{1}-\eqref{2}. Theorem \ref{th1.3} is proved.\end{proof}

\section*{Acknowledgements}
The second named author is financially supported by a grant from
the Ministry of Science and Education of the Republic of
Kazakhstan (Grant No. 0820/GF4).


\begin{thebibliography}{00}

%% \bibitem must have the following form:
\bibitem{hadamard23} J.~Hadamard, Lectures on the Cauchy Problem in
Linear Differential Equations, Yale University Press, New Haven,
CT, (1923).

\bibitem{carleman26} T.~Carleman, Les Fonctions Quasianalytiques, Gauthier-Villars, Paris (1926).

\bibitem{pucci55}	C.~Puññi, Sui problema di Cauchy non "ben posti". {\it Atti Accad. Naz. d. Lincei.} V. 18, No. 5. P.473-477, (1955).

\bibitem{lavrentev55} M.~M.~Lavrentev, On a Cauchy Problem for the Poisson Equation. {\it Izv. AS USSR, ser. Math.} V. 20, No. 6. P.819-842, (1955).

\bibitem{mergelyan56}	S.~N.~Mergelyan, Harmonic Approximation and Approximate Solution of the Cauchy Problem for the Laplace Equation. {\it Usp. Matem. Nauk.} V. 11, No. 5. P.3-26, (1956).

\bibitem{lavrentev57}	M.~M.~Lavrentev, On the Cauchy problem for second order linear elliptic equations. {\it Dokl. Akad. Nauk.} V. 112, P.195-197, (1957).

\bibitem{landis63}	M.~Landis, Some problems of the qualitative theory of second order elliptic equations, {\it Russian Math. Surveys.} V. 18, P.1-62, (1963).

\bibitem{tikhonovarsenin79} A.~N.~Tikhonov, V.~Y.~Arsenin, Methods for solving ill-posed problems. Moskow: Nauka, (1979).

\bibitem{qianfuli08}	Z.~Qian, C.~L.~Fu, Z.~P.~Li, Two regularization methods for a Cauchy problem for the Laplace equation. {\it J. Math. Anal. Appl.} V. 338, No. 1. P.479-489, (2008).

\bibitem{falkmonk86}	R.~S.~Falk, P.~B.~Monk, Logarithmic convexity for discrete harmonic functions and the approximation of the Cauchy problem for Poisson's equation, {\it Math. Comp.} V. 47, P.135-149, (1986).

\bibitem{kabanikhin08}	S.~I.~Kabanikhin, Definitions and examples of inverse and ill-posed problems. {\it J. Inv. Ill-Posed Problems.} V. 16, P.317-357, (2008).

\bibitem{eldensimoncini09}	L.~Eld\'{e}n, V.~Simoncini, A numerical solution of a Cauchy problem for an elliptic equation by Krylov subspaces. {\it Inverse Probl.} V. 25, No. 6, P.2187-2205, (2009).

\bibitem{sun} Y.~Sun, Modified method of fundamental solutions for the Cauchy problem connected with the Laplace equation. {\it International Journal of Computer Mathematics.} V. 91, No. 10. P.2185-2198, (2014).

\bibitem{klibanov15}	M.~V.~Klibanov, Carleman estimates for the regularization of ill-posed Cauchy problems. {\it Applied Numerical Mathematics.} V. 94, P.46-74, (2015).

\bibitem{kalmenoviskakova07}	T.~S.~Kal'menov, U.~A.~Iskakova, A criterion for the strong solvability of the mixed cauchy problem for the Laplace equation. {\it Doklady Mathematics.} V. 75, No. 3. P.370-373, (2007).

\bibitem{kalmenoviskakova08}	T.~S.~Kal'menov, U.~A.~Iskakova, A method for solving the Cauchy problem for the Laplace equation. {\it Doklady Mathematics.} V. 78, No. 3. P. 874-876, (2008).

\bibitem{kalmenoviskakova09}	T.~S.~Kal'menov, U.~A.~Iskakova, Criterion for the strong solvability of the mixed Cauchy problem for the Laplace equation. {\it Differential Equations.}  V. 45, No. 10. P. 1460-1466 (2009).

\bibitem{gohberg03} I. Gohberg, S. Goldberg, and Marinus A. Kaashoek, Basic Classes of Linear Operators. Birkhauser, Verlag, (2003).

\bibitem{naimark68}	M.~A.~Naimark, Linear Differential Operators, Part II. Ungar, New York, (1968).

\bibitem{titchmarsh39}	E.~.C.~Titchmarsh, The Theory of Functions. 2nd. Oxford University Press, (1939).

\end{thebibliography}
\end{document}